\RequirePackage{fixltx2e}
\documentclass[oneside,english]{amsart}
\usepackage[T1]{fontenc}
\usepackage[latin9]{inputenc}
\usepackage{babel}
\usepackage{textcomp}
\usepackage{amstext}
\usepackage{amsthm}
\usepackage{amssymb}
\usepackage{enumerate}
\usepackage[all]{xy}
\usepackage[unicode=true,
 bookmarks=true,bookmarksnumbered=false,bookmarksopen=false,
 breaklinks=false,pdfborder={0 0 1},backref=false,colorlinks=false]
 {hyperref}
\hypersetup{pdftitle={Finite subgroups of the birational automorphism group are `almost' nilpotent},
 pdfauthor={Attila Guld},
 pdfsubject={Algebraic Geometry}}

\makeatletter
\numberwithin{equation}{section}
\numberwithin{figure}{section}
\theoremstyle{plain}
\newtheorem{thm}{\protect\theoremname}
  \theoremstyle{plain}
  \newtheorem{prop}[thm]{\protect\propositionname}
  \theoremstyle{plain}
  \newtheorem{lem}[thm]{\protect\lemmaname}
  \theoremstyle{plain}
  
  \theoremstyle{remark}
  \newtheorem{rem}[thm]{\protect\remarkname}
  \theoremstyle{definition}
  
  \theoremstyle{plain}
  
  \theoremstyle{definition}
  \newtheorem{defn}[thm]{\protect Definition}
\usepackage{fancyhdr}
\usepackage{ifpdf} 
\ifpdf 
 \IfFileExists{lmodern.sty}{\usepackage{lmodern}}{}
\fi

\makeatother

  \providecommand{\corollaryname}{Corollary}
  \providecommand{\examplename}{Example}
  \providecommand{\lemmaname}{Lemma}
  \providecommand{\propositionname}{Proposition}
  \providecommand{\questionname}{Question}
  \providecommand{\remarkname}{Remark}
  \providecommand{\theoremname}{Theorem}

\DeclareMathOperator{\Spec}{Spec}
\DeclareMathOperator{\Aut}{Aut}
\DeclareMathOperator{\GL}{GL}

\DeclareMathOperator{\Ker}{Ker}
\DeclareMathOperator{\Imag}{Im}

\DeclareMathOperator{\Bir}{Bir}
\DeclareMathOperator{\KL}{\Gamma L}
\DeclareMathOperator{\Sym}{Sym}
\DeclareMathOperator{\Z}{Z}
\DeclareMathOperator{\Cent}{C}
\DeclareMathOperator{\H0}{H^0}
\DeclareMathOperator{\Diff}{Diff}
\DeclareMathOperator{\T}{T}
\begin{document}

\title[The birational automorphism group is nilpotently Jordan]{Finite subgroups of the birational automorphism group are `almost' nilpotent}

\author{Attila Guld}
\email{guld.attila@renyi.mta.hu}

\thanks{The research was partly supported by the National Research, Development and Innovation Office (NKFIH) Grant No. K120697. 
The project leading to this application has received funding from the European Research Council (ERC) under the European Union's Horizon 2020 research and innovation programme (grant agreement No 741420).}

\address{
R\'enyi Alfr\'ed Matematikai Kutat\'oint\'ezet\\
Re\'altanoda utca 13-15.\\
Budapest, H1053\\
Hungary}

\begin{abstract}
We call a group $G$ nilpotently Jordan of class at most $c$ $(c\in\mathbb{N})$ 
if there exists a constant $J\in\mathbb{Z}^+$ such that every finite subgroup $H\leqq G$ contains a nilpotent subgroup $K\leqq H$ of class at most $c$ and index at most $J$.\\
We show that the birational automorphism group of a $d$ dimensional variety over a field of characteristic zero is nilpotently Jordan of class at most $d$. 
\end{abstract}

\maketitle

\section{Introduction}

\begin{defn}
\label{nilpJord}
A group $G$ is called Jordan, solvably Jordan or nilpotently Jordan of class at most $c$ ($c\in\mathbb{N}$) if there exists a constant $J=J(G)\in\mathbb{Z}^+$, only depending on $G$,
 such that every finite subgroup $H\leqq G$ has a subgroup $K\leqq H$ such that
$|H:K|\leqq J$ and $K$ is Abelian, solvable or nilpotent of class at most $c$, respectively.
\end{defn}

The notion of Jordan groups and solvably Jordan groups was introduced by V. L. Popov (Definition 2.1 in \cite{Po11}) and Yu. Prokhorov and C. Shramov (Definition 8.1 in \cite{PS14}), respectively.


\begin{thm}
\label{main}
The birational automorphism group of a $d$ dimensional variety over a field of characteristic zero is 
nilpotently Jordan of class at most $d$.
\end{thm}

\begin{rem}
\label{C}
It is enough to prove the theorem over the field of the complex numbers. Indeed, let $K$ be a field of characteristic zero and $X$ be a variety over $K$. 
We can fix a finitely generated field extension $L_0|\mathbb{Q}$ and an $L_0$-variety $X_0$ such that $X\cong X_0\times_{L_0}\Spec K$. 
Fix a field embedding $L_0\hookrightarrow\mathbb{C}$ and let $X^*\cong X_0\times_{L_0}\Spec\mathbb{C}$.
For an arbitrary finite subgroup $G\leqq\Bir(X)$ we can find a finitely generated field extension $L_1|L_0$ such that the elements of $G$ can be defined as birational transformations over the field $L_1$. Hence
 $G\leqq \Bir(X_1)$, where $X_1\cong X_0\times_{L_0}\Spec L_1$. 
We can extend the fixed field embedding $L_0\hookrightarrow\mathbb{C}$ to a field embedding $L_1\hookrightarrow\mathbb{C}$.
Therefore $X^*\cong X_0\times_{L_0}\Spec\mathbb{C}\cong X_1\times_{L_1}\Spec\mathbb{C}$, and we can embed $G$ to the birational automorphism group of the complex variety $X^*$.
As the birational class of the complex variety $X^*$ only depends on the birational class of the variety $X$, it is enough to examine complex varieties.
\end{rem}
 
In the following discussion we shortly sketch the history of Jordan type properties in birational geometry over fields of \textit{characteristic zero}. 
Research about investigating the Jordan property of the birational automorphism group of a variety was initiated by J.-P. Serre (\cite{Se09}) and V. L. Popov (\cite{Po11}).
 In \cite{Se09} J.-P. Serre settled the problem for the Cremona group of rank two (by showing that it enjoys the Jordan property), 
while  in the articles \cite{Po11}, \cite{Za15} V. L. Popov and Yu. G. Zarhin solved the question for one and two dimensional varieties.
They found that the birational automorphism group of a curve or a surface is Jordan, save when the variety is birational to a direct product of an elliptic curve and the projective line. 
This later case was examined in \cite{Za15}, where -based on calculations of D. Mumford- the author was able to conclude that the birational automorphism group contains Heisenberg $p$-groups for arbitrarily large prime numbers $p$. 
Hence it does not enjoy the Jordan property.\\ 
In \cite{PS14} and \cite{PS16} Yu. Prokhorov and C. Shramov made important contributions to the subject using the arsenal of the Minimal Model Program and assuming the Borisov-Alexeev-Borisov (BAB) conjecture 
(which has later been verified in the celebrated article \cite{Bi16} of C. Birkar;  for a survey paper on the work of C. Birkar and its connection to the Jordan property, the interested reader can consult with \cite{Ke19}). 
Amongst many highly interesting results, Yu. Prokhorov and C. Shramov proved that the birational automorphism group of a rationally connected variety and 
the birational automorphism group of a non-uniruled variety is Jordan. To answer a question of D. Allcock, they also introduced the concept of solvably Jordan groups, 
and showed that the birational automorphism group of an arbitrary variety is solvably Jordan.\\
The landscape is strikingly similar in differential geometry. The techniques are fairly different, still the results converge to similar directions. 
In the following we briefly review the history of the question of Jordan type properties of diffeomorphism groups of smooth compact real manifolds. (We note that there are many other interesting setups which were considered by differential geometers; 
for a very detailed account see the Introduction of \cite{MR18}.) As mentioned in \cite{MR18}, during the mid-nineties  \'E. Ghys conjectured that the diffeomorphism group of a smooth compact real manifold is Jordan, 
and he proposed this problem in many of his talks (\cite{Gh97}). The case of surfaces follows from the Riemann-Hurwitz formula (see \cite{MR10}), the case of 3-folds are more involved. 
In \cite{Zi14} B. P. Zimmermann proved the conjecture for them using the geometrization of compact 3-folds (which follows from the work of W. P. Thurston and G. Perelman). 
I. Mundet i Riera also verified the conjecture for several interesting cases, like tori, projective spaces, homology spheres and manifolds with non-zero Euler characteristic (\cite{MR10},\cite{MR16}, \cite{MR18}).\\ 
However, in 2014, B. Csik\'os, L. Pyber and E. Szab\'o found a counterexample (\cite{CPS14}). 
Their construction was remarkably analogous to the one of Yu. G. Zarhin. They showed that if the manifold $M$ is diffeomorphic to the direct product of the two-sphere and the two-torus 
or to the total space of any other smooth orientable two-sphere bundle over the two-torus, then the diffeomorphism group contains Heisenberg $p$-groups for arbitrary large prime numbers $p$. Hence $\Diff(M)$ cannot be Jordan.
As a consequence, \'E. Ghys improved on his previous conjecture, and proposed the problem of showing that the diffeomorphism group of a compact real manifold is nilpotently Jordan (\cite{Gh15}). 
As the first trace of evidence,  I. Mundet i Riera and  C. Sa\'ez-Calvo showed that the diffeomorphism group of a 4-fold is nilpotently Jordan of class at most 2 (\cite{MRSC19}). Their proof uses results from the classification theorem of finite simple groups.\\
Motivated by these antecedents, in this article we investigate the nilpotently Jordan property for birational automorphism groups of varieties.\\

The idea of the proof stems from the following picture. Let $X$ be a $d$ dimensional complex variety. We can assume that $X$ is smooth and projective. Let $G\leqq\Bir(X)$ be an arbitrary finite subgroup.  
Consider the MRC (maximally rationally connected) fibration $\phi:X\dashrightarrow Z$ (Theorem \ref{MRC}). Because of the functoriality of the MRC fibration, a birational $G$-action is induced on $Z$, making $\phi$ $G$-equivariant. 
After a smooth regularization (Lemma \ref{reg}) we can assume that both $X$ and $Z$ are smooth and projective, $G$ acts on them by regular automorphisms  and  $\phi$ is a $G$-equivariant morphism. 
Since the general fibres of $\phi$ are rationally connected, we can run a $G$-equivariant relative Minimal Model Program over $Z$ on $X$ (Theorem \ref{MMP}). It results a $G$-equivariant Mori fibre space $\varrho:W\to Y$ over $Z$.
 \[
\xymatrix{
X\ar@ {-->} [r]^{\cong} \ar[rd]_{\phi} & W \ar[r]^\varrho \ar[d] & Y \ar[ld]^\psi\\
& Z
}
\]
We can understand the $G$-action on $X$ by analyzing the $G$-actions on $\psi:Y\to Z$ and on $\varrho:W\to Y$. 
We will apply induction on the relative dimension $e=\dim X-\dim Z$ to achieve this (Theorem \ref{AlmostMain}). 
Actually, we will prove a slightly stronger theorem then Theorem \ref{main} and will show that $\Bir(X)$ is nilpotently Jordan of class at most $(e+1)$.
The base of the induction is when $e=0$. Then $X$ is non-uniruled and a theorem of Yu. Prokhorov and C. Shramov (Theorem 1.8 in \cite{PS14}) shows us that the birational automorphism group of $X$ is Jordan.\\
Otherwise, the inductive hypothesis will show us that $H=\Imag(G\to\Aut_{\mathbb{C}}(Y))$ has a bounded index nilpotent subgroup of class at most $e$.
To perform the inductive step, we will take a closer look at the $G$-action on the generic fibre $W_\eta\to\Spec K(Y)$. We will use two key ingredients. 
The first one is based on the boundedness of Fano varieties, and will allow us to embed $G$ into the semilinear group $\GL(n, K(Y))\rtimes\Aut_{\mathbb{C}}(K(Y))$, where $n$ is bounded in terms of $e$ (Proposition \ref{Fano}). 
The second one is a Jordan type theorem on certain finite subgroups of a semilinear group (Theorem \ref{groupmain}).
Putting these together will finish the proof.\\  

 The article is organized in the following way. In Section \ref{P} we recall the definition and some basic facts about nilpotent groups, we also recall the concept of the MRC fibration. 
 In Section \ref{FGV} we collect results about finite birational group actions on varieties. 
 In particular, it contains the theorem of Yu. Prokhorov and C. Shramov about the Jordan property of the birational automorphism group of non-uniruled and rationally connected varieties (Theorem \ref{nu}), 
 the regularization lemma (Lemma \ref{reg}), the theorem on the $G$-equivariant MMP (Theorem \ref{MMP}) and
 the proposition about certain finite group actions on Fano varieties (Proposition \ref{Fano}). At the end of the section we investigate some questions about bounds on the number of generators of finite subgroups of the birational automorphism group.
 The boundedness of the generating set helps us to give a more accurate bound on the nilpotency class (Remark \ref{NoB}).
 Section \ref{gp} deals with the proof of the Jordan type theorem on semilinear groups (Theorem \ref{groupmain}). 
 Finally, in Section \ref{PMT} we prove our main theorem.
 
\subsection*{Acknowledgements}
The author is very grateful to E. Szab\'o for many helpful discussions.

\section{Preliminaries}
\label{P}
\subsection{Nilpotent groups}
We recall the definition of nilpotent groups and some of their basic properties.

\begin{defn}
Let $G$ be a group.
Let $\Z_0(G)=1$ and define $\Z_{i+1}(G)$ as the preimage of $\Z(G/\Z_i(G))$ under the natural quotient group homomorphism $G\to G/\Z_i(G)$ $(i\in\mathbb{N})$. The series of groups
$1=\Z_0(G)\leqq\Z_1(G)\leqq\Z_2(G)\leqq...$ is called the upper central series of $G$.\\
Let $\gamma_0(G)=G$ and let $\gamma_{i+1}(G)=[\gamma_i(G),G]$ ($i\in\mathbb{N}$, and $[,]$ denotes the commutator operation). The series of groups
$G=\gamma_0(G)\geqq\gamma_1(G)\geqq\gamma_2(G)\geqq...$ is called the lower central series of $G$.\\
$G$ is called nilpotent if one (hence both) of the following equivalent conditions hold:
\begin{itemize}
\item
There exists $n\in\mathbb{N}$ such that $\Z_n(G)=G$.
\item
There exists $n\in\mathbb{N}$ such that $\gamma_n(G)=1$.
\end{itemize} 
If $G$ is a nontrivial nilpotent group, then there exists a natural number $c$ for which  $\Z_c(G)=G$, $\Z_{c-1}(G)\neq G$ and  $\gamma_c(G)=1$, $\gamma_{c-1}(G)\neq 1$ holds. $c$ is called the nilpotency class of $G$. 
(If  $G$ is trivial, then its nilpotency class is zero.)
\end{defn}

\begin{rem}
Note that $\Z_1(G)$ is the centre of the group $G$, while $\gamma_1(G)$ is the commutator subgroup. A non-trivial group $G$ is nilpotent of class one if and only if it is Abelian.\\
Nilpotency is the property between the Abelian and the solvable properties. The Abelian property implies nilpotency, while nilpotency implies solvability.
\end{rem}

The following proposition describes one of the key features of nilpotent groups. They can be built up by successive central extensions.
\begin{prop}
\label{CE}
Let $G$ be a group and $A\leqq\Z(G)$ be a central subgroup of $G$. If $G/A$ is nilpotent of class at most $c$, then $G$ is nilpotent of class at most $(c+1)$.
\end{prop}

We will use also the two properties below about nilpotent groups.
\begin{prop}
\label{ICmap}
Let $G$ be a nilpotent group of class at most $n$. 
Fix $n-1$ arbitrary elements in $G$, denote them by $g_1,g_2,...g_{n-1}$, and let $1\leqq j \leqq n$ be an arbitrary integer.
The map $\varphi_j$ defined by the help of iterated commutators of length $(n-1)$
\begin{gather*}
\varphi_j:G\to\gamma_{n-1}(G)\\
g\mapsto [[...[[[[...[[g_1,g_2],g_3]...],g_{j-1}],g],g_j]...],g_{n-1}]
\end{gather*} 
gives a group homomorphism.
\end{prop}

\begin{prop}
\label{IC}
Let $G$ be a group. $G$ is nilpotent of class at most $n$ if and only if $\forall g_1,g_2,...,g_{n+1}\in G$: $[[...[[g_1,g_2],g_3]...],g_{n+1}]=1$.
\end{prop}

\begin{rem}
Typical examples of nilpotent groups are finite $p$-groups (where $p$ is a prime number). If we restrict our attention to finite nilpotent groups, even more can be said. 
(Recall that a $p$-Sylow subgroup of a finite group is the largest $p$-group contained in the group.)
A finite group is nilpotent if and only if it is the direct product of its Sylow subgroups (Theorem 6.12 in \cite{CR62}).
\end{rem}

\subsection{The maximally rationally connected fibration}

We recall the concept of the maximally rationally connected fibration. For a detailed treatment see Chapter $4$ of \cite{Ko96}, for the non-uniruledness of the basis see Corollary 1.4 in \cite{GHS03}.

\begin{thm}
\label{MRC}
Let $X$ be a smooth proper complex variety. The pair $(Z,\phi)$ is called the maximally rationally connected (MRC) fibration if 
\begin{itemize}
\item $Z$ is a complex variety,
\item $\phi:X\dashrightarrow Z$ is a dominant rational map,
\item there exist open subvarieties $X_0$ of $X$ and $Z_0$ of $Z$ such that $\phi$ descends to a proper morphism between them $\phi_0:X_0\to Z_0$ with rationally connected fibres,
\item if $(W,\psi)$ is another pair satisfying the three properties above, then $\phi$ can be factorized through $\psi$. More precisely, there exists a rational map $\tau: W\dashrightarrow Z$ such that $\phi=\tau\circ\psi$. 
\end{itemize}
The MRC fibration exists and is unique up to birational equivalence. Moreover the basis $Z$ is non-uniruled.
\end{thm}

\section{Finite group actions on varieties}
\label{FGV}

In this section we introduce techniques which help us to solve partial cases of our problem and help us to build up the full solution from the special cases.\\

\subsection{Jordan property}

Yu. Prokhorov and C. Shramov proved the following theorem (Theorem 1.8 in \cite{PS14}  and Theorem 1.8 in \cite{PS16}). 
It will serve us as a starting point of an inductive argument in the proof of our main theorem and 
will be an important ingredient when we look for bounds on the number of generators of finite subgroups of the birational automorphism group (Theorem \ref{bfsg}).

\begin{thm}
\label{nu}
Let $X$ be variety over a field of characteristic zero. Assume that $X$ is either non-uniruled or rationally connected. Then the birational automorphism group of $X$ is Jordan (in other words, it is nilpotently Jordan of class at most 1). 
\end{thm}

\subsection{Smooth regularization}

The next lemma is a slight extension of the well-known (smooth) regularization of finite group actions on varieties  (Lemma-Definition 3.1. in\cite{PS14}).

\begin{lem}
\label{reg}
Let $X$ and $Z$ be complex varieties and $\phi:X\dashrightarrow Z$ be a dominant rational map between them. Let $G$ be a finite group which acts by birational automorphisms on $X$ and $Z$ in such a way that $\phi$ is $G$-equivariant. 
There exist smooth projective varieties 
$X^*$ and $Z^*$ with regular $G$-actions on them and a $G$-equivariant projective morphism $\phi^*: X^*\to Z^*$ such that 
$X^*$ is $G$-equivariantly birational to $X$, $Z^*$ is $G$-equivariantly birational to $Z$ and $\phi^*$ is $G$-equivariantly birational to $\phi$. In other words, we have a $G$-equivariant commutative diagram.
\[
\xymatrix{
X \ar@{-->}[r]^\cong \ar@{-->}[d]^\phi & X^* \ar[d]^{\phi^*}\\
Z \ar@{-->}[r]^\cong & Z^*
}
\]

\end{lem}
\begin{proof}
Let $K(Z)\leqq K(X)$ be the field extension corresponding to the function fields of $Z$ and $X$, induced by $\phi$. 
Take the induced $G$-action on this field extension and let $K(Z)^G\leqq K(X)^G$ be the field extension of the $G$-invariant elements. 
Consider a projective model of it, i.e. let $\varrho_1: X_1\to Z_1$ be a (projective) morphism, where $X_1$ and $Z_1$ are projective varieties such that $K(X_1)\cong K(X)^G$ and $K(Z_1)\cong K(Z)^G$, 
and $\varrho_1: X_1\to Z_1$ induces the field extension  $K(Z_1)\cong K(Z)^G\leqq K(X)^G\cong K(X_1)$.
By normalizing $X_1$ in the function field $K(X)$ and $Z_1$ in the function field $K(Z)$ we get projective varieties $X_2$ and $Z_2$, moreover $\varrho_1$ induces a $G$-equivariant morphism $\varrho_2:X_2\to Z_2$ between them.\\
As the next step, we can take a $G$-equivariant resolution of singularities $\widetilde{Z_2}\to Z_2$. After replacing $Z_2$ by $\widetilde{Z_2}$ 
and $X_2$ by the irreducible component of $X_2\times_{Z_2}\widetilde{Z_2}$ which dominates $\widetilde{Z_2}$, we can assume that $Z_2$ is smooth. Hence $G$-equivarianlty resolving the singularities of $X_2$ finishes the proof.
\end{proof}

\subsection{Minimal Model Program and boundedness of  Fano varieties}

Applying the results of the famous article by C. Birkar, P. Cascini, C. D. Hacon and J. McKernan (\cite{BCHM10}) enables us to use the arsenal of the Minimal Model Program. 
As a consequence, we can examine rationally connected varieties (fibres) with the help of Fano varieties (fibres). 
For the later we can use boundedness results because of yet another famous theorem  by C. Birkar (\cite{Bi16}). (This theorem was previously known as the BAB Conjecture).

\begin{thm}
\label{MMP}
Let $X$ and $Z$ be smooth projective complex varieties such that $\dim Z<\dim X$. Let $\phi:X\to Z$ be a dominant morphism between them with rationally connected general fibres. 
Let $G$ be a finite group which acts by regular automorphisms on $X$ and $Z$ in such a way that $\phi$ is $G$-equivariant. 
We can run a $G$-equivariant Minimal Model Program (MMP)
on $X$ relative to $Z$ which results a Mori fibre space. In particular, the Minimal Model Program gives a $G$-equivariant commutative diagram 
\[
\xymatrix{
X\ar@ {-->} [r]^{\cong} \ar[rd]^{\phi} & W \ar[r] \ar[d] & Y \ar[ld]\\
& Z
}
\]
where $W$ is $G$-equivariantly birational to $X$, $\dim Y< \dim X$ and the generic fibre of the morphism between $W$ and $Y$ is  a Fano variety with (at worst) terminal singularities.
\end{thm}
\begin{proof}
By Corollary 1.3.3 of \cite{BCHM10}, we can run a relative MMP on $\phi:X\to Z$ (which results a Mori fibre space) if the canonical divisor of $X$ is not $\phi$-pseudo-effective. It can be done equivariantly if we have 
finite group actions. (See Section 2.2 in \cite{KM98} and Section 4 of \cite{PS14} for further discussions on the topic.) So, it remains to show that the canonical divisor of $X$ is not $\phi$-pseudo-effective.\\
By generic smoothness, a general fibre of $\phi$ is a smooth rationally connected projective complex variety. 
Therefore if $x$ is a general closed point of a general fibre $F$, then there exists a free rational curve $C_x$ running through $x$, lying entirely in the fibre $F$ (Theorem 1.9 of Chapter 4 in \cite{Ko96}). 
Since $C_x$ is a free rational curve, $C_x.K_X\leqq-2$. Since the inequality holds for every general closed point of every general fibre, $K_X$ cannot be $\phi$-pseudo-effective.
\end{proof}

The lemmas and the theorems above open the door for us to use induction on the relative dimension of the MRC fibration while proving Theorem \ref{main}.  So we only need to deal with Fano varieties of bounded dimensions.

\begin{prop}
\label{Fano}
Let $e$ be a natural number. There exists a constant $n=n(e)\in\mathbb{N}$, only depending on $e$, with the following property. If
\begin{itemize}
\item $K$ is a field of characteristic zero,
\item $F$ is a Fano variety over $K$ of dimension at most $e$, with terminal singula/-rities,
\item $G$ is a finite group which acts  faithfully on $F$ by regular automorphisms of the $\mathbb{Q}$-scheme $F$, 
and acts on $\Spec K$ by regular automorphisms of the $\mathbb{Q}$-scheme $\Spec K$, 
in such a way that the structure morphism $F\to\Spec K$ is $G$-equivariant,
\end{itemize}
then $G$ can be embedded into the semilinear group $\KL (n, K)\cong \GL(n, K)\rtimes \Aut K$ 
in such a way that $G\hookrightarrow \KL (n,K)\twoheadrightarrow\Aut K$ corresponds to the $G$-action on $\Spec K$. 
\end{prop}

\begin{proof}
Fix $K$, $F$ and $G$ with the properties described by the theorem. There exists a finitely generated field extension $L_0|\mathbb{Q}$ and a  Fano variety $F_0$ over $L_0$ such that $F\cong F_0\times_{L_0}\Spec K$.
Consider an embedding of fields $L_0\hookrightarrow\mathbb{C}$, and let $F_1\cong F_0\times_{L_0}\Spec \mathbb{C}$.
Since complex Fano varieties with terminal singularities of bounded dimension form a bounded family (Theorem1.1 in\cite{Bi16}), there exist constants $P=P(e),M=M(e)\in\mathbb{N}$, only depending on $e$,
such that $P$-th power of the anticanonical divisor embeds $F_1$ to the $M_1$-dimensional complex projective space, where $M_1\leqq M$.
Since the $P$-th power of the anticanonical divisor is defined over any field, this embedding is defined over any field, in particularly over $K$. 
So we have a closed embedding of the form $F\hookrightarrow \mathbb{P}_K^{M_1}\cong\mathbb{P}(\H0 (X,-K_F^P)^*)$.\\ 
By the functorial property of a (fixed) power of the anticanonical divisor, an equivariant $G$-action is induced on the commutative diagram below.
\[
\xymatrix{
F\ar@{^{(}->}[r] \ar[d] & \mathbb{P}(\H0 (X,-K_F^P)^*) \ar[ld] \\
\Spec K 
}
\]
Since $F\hookrightarrow\mathbb{P}(\H0 (X,-K_F^P)^*)$ is a closed embedding, the semilinear action of $G$ on the vector space $\H0 (X,-K_F^P)$ is faithful. 
Hence $G$ embeds to $\KL(\H0 (X,-K_F^P))$. Clearly $G\to \Aut K$ corresponds to the $G$-action on $\Spec K$. 
As $\dim\H0 (X,-K_F^P)\leqq M(e)+1$, we finished the proof.
\end{proof}

\subsection{Bound on the number of generating elements of finite subgroups of the birational automorphism groups}

Now we turn our attention on finding bounds on the number of generating elements of finite subgroups of the birational automorphism group of varieties. 
It will be important for as when we will investigate commutator relations (Lemma \ref{DN}), and it will be crucial to have a bound on the number of the elements of a generating set of the group.\\
The next theorem and its proof are essentially due to Y. Prokhorov and C. Shramov. (We use the world essentially as they only considered the case of finite Abelian subgroups (Remark 6.9 of \cite{PS14}).)
It is also important to note that the proof of Remark 6.9 of \cite{PS14} uses the result of C. Birkar about the boundedness of Fano varieties (Theorem 1.1 in \cite{Bi16}).

\begin{thm}
\label{bfsg}
Let $X$ be a variety over a field of characteristic zero. There exists a constant $m=m(X)\in\mathbb{Z}^+$, only depending on the birational class of $X$, such that
if $G\leqq\Bir(X)$ is an arbitrary finite subgroup of the birational automorphism group, then $G$ can be generated by $m$ elements.
\end{thm}
\begin{proof}
First we show the theorem in the special cases when $X$ is either non-uniruled or rationally connected. 
By Remark 6.9 of \cite{PS14} and Theorem 1.1 of \cite{Bi16}, there exists a constant $m=m(X)\in\mathbb{Z}^+$, only depending on the birational class of $X$, such that
if $A\leqq\Bir(X)$ is an arbitrary finite Abelian subgroup of the birational automorphism group, then $A$ can be generated by $m$ elements. Since $\Bir(X)$ is Jordan when $X$ is non-uniruled or rationally connected (Theorem \ref{nu}),
the result on the finite Abelian groups implies the claim of the theorem in both of these special cases.\\
Now let $X$ be arbitrary. Arguing as in Remark \ref{C} we can assume that $X$ is a complex variety. 
Consider the MRC fibration $\phi:X\dashrightarrow Z$.
By Lemma \ref{reg} we can assume that both $X$ and $Z$ are smooth projective varieties, and $G$ acts on them by regular automorphisms. 
Let $\rho$ be the generic point of $Z$, and let $X_\rho$ be the generic fibre of $\phi$. $X_\rho$ is a rationally connected variety over the function field $k(Z)$.\\
Let $G_\rho\leqq G$ be the maximal subgroup of $G$ acting fibrewise. $G_\rho$ has a natural faithful action on $X_\rho$, while $G/G_\rho=G_Z$ has a natural faithful action on $Z$. 
This gives a short exact sequence of groups
\[1\to G_\rho\to G\to G_Z\to 1.\]
By the rationally connected case there exists a constant $m_1(X_{\rho})$, only depending on the birational class of $X_{\rho}$, such that $G_\rho$ can be generated by $m_1(X_{\rho})$ elements. 
By the non-uniruled case there exists a constant $m_2(Z)$, only depending on the birational class of $Z$, such that $G_Z$ can be generated by $m_2(Z)$ elements.
So $G$ can be generated by $m(X_{\rho}, Z)=m_1(X_{\rho})+m_2(Z)$ elements. Since $m(X_{\rho},Z)$ only depends on the birational classes of $X_{\rho}$ and $Z$, 
and both of the birational classes of $X_{\rho}$ and $Z$ only depend on the birational class of $X$, this finishes the proof.
\end{proof}

In case of rationally connected varieties we will use a slightly stronger version of the theorem. To prove it, we need a theorem about fixed points of rationally connected varieties. 
It is due to Yu. Prokhorov and C. Shramov (Theorem 4.2 of \cite{PS14}).

\begin{thm}
\label{afp}
Let $e$ be a natural number. There exits a constant $R=R(e)\in\mathbb{Z}^+$, only depending on $e$, with the following property.
If $X$ is a rationally connected complex projective variety of dimension at most $e$,
and $G\leqq\Aut(X)$ is an arbitrary finite subgroup of its automorphism group,
then there exists a subgroup $H\leqq G\leqq \Aut(X)$ such that $H$ has a fixed point in $X$, and the index of $H$ in $G$ is bounded by $R$.
\end{thm}

\begin{thm}
\label{bgrc}
Let $e$ be a natural number. There exits a constant $m=m(e)\in\mathbb{Z}^+$, only depending on $e$, with the following property.
If $K$ is an arbitrary field of characteristic zero, $X$ is a rationally connected variety over $K$ of dimension at most $e$,
and $G\leqq\Bir(X)$ is an arbitrary finite subgroup of the birational automorphism group,
then $G$ can be generated by $m$ elements.
\end{thm}
\begin{proof}
Fix $K$, $X$ and $G$ with the properties described by the theorem.
Arguing as in the case of Remark \ref{C}, we can assume that $K$ is the field of the complex numbers.\\
Using Lemma \ref{reg}, we can assume that $X$ is smooth and projective and $G$ is a finite subgroup of the biregular automorphism group $\Aut(X)$.\\
By Theorem \ref {afp}, we can assume that $G$ has a fixed point in $X$. Denote it by $P$.\\
By Lemma 4 of \cite{Po14} $G$ acts faithfully on the tangent space of the fixed point $P$. So $G$ can be embedded to $\GL(\T_PX)$, whence $G$ can be embedded to $\GL(e,\mathbb{C})$.
Therefore the claim of the theorem follows from Lemma \ref{bg}. This finishes the proof.
\end{proof}

\section{Calculations in the general semilinear group}
\label{gp}

This section contains the group theoretic ingredient of the proof of the main theorem.

\begin{thm}
\label{groupmain}
Let $c,n$ and $m$ be positive integers. Let $F$ be the family of those finite groups $G$ which have the following properties. 
\begin{itemize}
\item
There exists a field $K$ of characteristic zero containing all roots of unity such that $G$ is a subgroup of the semilinear group $\KL(n,K)\cong \GL(n)\rtimes \Aut K$. 
\item
Every subgroup of $G$ can be generated by $m$ elements.
\item
The image of the composite group homomorphism $G\hookrightarrow \KL(n,K)\twoheadrightarrow\Aut K$, denoted by $\Gamma$, is nilpotent of class at most $c$ ($c \in\mathbb{N}$) and fixes all roots of unity.
\end{itemize}
There exists a constant $C=C(c,n,m)\in\mathbb{Z}^+$, only depending on $c,n$ and $m$, such that every finite group $G$ belonging to $F$ contains a nilpotent subgroup $H\leqq G$ with nilpotency class at most $(c+1)$ and with index at most $C$. 
\end{thm}

First, we recall a slightly strengthened version of Jordan's theorem.
\begin{thm}
\label{Jor}
Let $n$ be a positive integer. There exists a constant $J=J(n)\in\mathbb{Z}^+$, only depending on $n$, 
such that if a finite group $G$ is a subgroup of a general linear group $\GL(n, K)$, where $K$ is a field of characteristic zero, then $G$ contains a characteristic Abelian subgroup $A\leqq G$ of index at most $J$. 
\end{thm}

\begin{rem}
The only claim of the above theorem which does not follow immediately from Theorem 2.3 in \cite{Br11} is that we require the Abelian subgroup of bounded index $A\leqq G$ to be characteristic (i.e. invariant under all automorphisms of $G$)
instead of being normal (i.e. invariant under the inner automorphisms of $G$).
In the following we will prove some lemmas which help us to deduce the above variant of the theorem from the one which can be found in \cite{Br11}.
\end{rem}

\begin{lem}
\label{bg}
Let $n$ be a positive integer. There exists a constant $r=r(n)\in\mathbb{Z}^+$, only depending on $n$, 
such that if a finite group $G$ is a subgroup of a general linear group $\GL(n, K)$, where $K$ is a field of characteristic zero, then $G$  can be generated by $r$ elements. 
\end{lem}
\begin{proof}
It is enough to prove the lemma when $K$ is algebraically closed, so we can assume it.
By Theorem 2.3 in \cite{Br11}, $G$ contains a diagonalizable subgroup of bounded index. 
Since finite diagonal groups of $\GL(n,K)$  can be generated by $n$ elements, the lemma follows. 
\end{proof}

\begin{lem}
\label{ind}
Let $J$ and $r$ be positive integers. There exists a constant $L=L(J,r)\in\mathbb{N}$, only depending on $r$ and $J$, such that
if $G$ is a finite group which can be generated by $r$ elements, then $G$ has at most $L$ many subgroups of index $J$.
\end{lem} 
\begin{proof}
Fix an arbitrary finite group $G$ which can be generated by $r$ elements. We can construct an injective map of sets 
from the set of index $J$ subgroups of $G$ to the set of group homomorphisms from $G$ to the symmetric group of degree $J$. 
Since $G$ can be generated by $r$ elements the later set has boundedly many elements, hence the former set has boundedly many elements as well. So we only left with the task of constructing such an injective map.\\
Let $S$ be a set with $J$ elements. We can identify the symmetric group of degree $J$, denoted by $\Sym_J$, with the symmetry group of the set $S$. Fix an arbitrary element $x\in S$.
For every index $J$ subgroup $K\leqq G$, fix a bijection $\mu_K$ between the set of the left cosets of $K$ and the set $S$,  subject to the following condition, $K$ is mapped to the fixed element $x$, i.e. $\mu_K(K)=x$.
Let $H\leqq G$ be an arbitrary subgroup of index $J$. $G$ acts on the set of the left cosets of $H$ by left multiplication. Using the bijection $\mu_H$, this induces a group homomorphism $\phi_H: G\to \Sym_J$ .
The constructed assignment is injective as the stabilizator subgroup of $x$ in the image group $\Imag\phi_H$ uniquely determines $H$.
\end{proof}

\begin{proof}[Proof of Theorem \ref{Jor}]
Let $K$ be an arbitrary field of characteristic zero, and let $G$ be an arbitrary finite subgroup of $\GL(n,K)$. 
By Theorem 2.3 in \cite{Br11}  $G$ contains an Abelian subgroup $A\leqq G$ of index bounded by $J_0=J_0(n)$. Consider the set $S$ of the smallest index Abelian subgroups of $G$.
By Lemma \ref{bg} and Lemma \ref{ind} there exists a constant $L=L(n)$, only depending on $n$, such that $S$ has at most $L$ many elements. 
Take the intersection of the subgroups contained in $S$, it gives a characteristic Abelian subgroup of index at most $J_0^L$.
\end{proof}

Next we prove a lemma about nilpotent groups.

\begin{lem}
\label{DN}
Let $c,J$ and $m$ be positive integers. There exists a constant $C=C(c,J,m)\in \mathbb{N}$, only depending on $c,J$ and $m$, such that if
\begin{itemize}
\item $G$ is a nilpotent group of class at most $(c+1)$, 
\item $G$ can be generated by $m$ elements, 
\item the cardinality of $\gamma_{c}(G)$ is at most $J$,
\end{itemize}
then $G$ has a nilpotent subgroup $H\leqq G$ of class at most $c$ whose index is bounded by $C$. 
\end{lem} 
\begin{proof}
Fix a generating system $g_1,...,g_m\in G$. Consider the group homomorphisms (Proposition \ref{ICmap})
\begin{gather*}
\varphi_{i_1,i_2,...,i_{c}}:G\to\gamma_{c}(G)\\
g\mapsto [[[...[[g_{i_1},g_{i_2}],g_{i_3}]...],g_{i_{c}}],g],
\end{gather*} 
where $1\leqq i_1,i_2,...,i_{c}\leqq m$, i.e. for every ordered length $c$ sequence of the generators we assign a group homomorphism using the iterated commutators.
Let $H$ be the intersection of the kernels.
$$H=\bigcap\limits_{1\leqq i_1,i_2,...,.i_{c} \leqq m} \Ker\varphi_{i_1,i_2,...,i_{c}}$$ 
Using the fact that the length $c$ iterated commutators give group homomorphisms in every variable if we fix the other variables (Proposition \ref{ICmap}), 
one can show that all the length $c$ iterated commutators of $H$ vanish. Hence $H$ is nilpotent of class at most $c$ (Proposition \ref{IC}).\\
On the other hand $H$ is the intersection of $m^c$ many subgroups of index at most $|\gamma_{c}(G)|\leqq J$. Hence the index of $H$ is bounded in terms of $c,J$ and $m$. This finishes the proof.
\end{proof}

Now we are ready to prove the main theorem of the section.

\begin{proof}[Proof of Theorem \ref{groupmain}]
Let  $K$ be an arbitrary field of characteristic zero containing all roots of unity, and let $G$ be an arbitrary finite subgroup of $\KL(n,K)$ belonging to $F$. Consider the short exact sequence of groups given by
$$1\to N\to G\to\Gamma\to 1$$
where $N=\GL(n,K)\cap G$ and $\Gamma=\Imag(G\to\Aut K)$. By Theorem \ref{Jor}, $N$ contains a characteristic Abelian subgroup of index bounded by $J=J(n)\in\mathbb{Z}^+$. 
Since $A$ is characteristic in $N$ and $N$ is normal in $G$, $A$ is a normal subgroup of $G$.\\
Consider the natural action of $G$ on the vector space $V=K^n$. Since $A$ is a finite Abelian subgroup of $\GL(V)$ and the ground field $K$ contains all roots of unity,
$A$ decomposes $V$ into common eigenspaces of its elements: $V=V_1\oplus V_2\oplus...\oplus V_r$ $(r\leqq n)$. 
As $A$ is normal in $G$, $G$ respects this decomposition, i.e. $G$ acts on the set of linear subspaces $\{V_1,V_2,...,V_r\}$ by permutations.
The kernel of this group action, denoted by $G_1$, is a bounded index subgroup of $G$ (indeed $|G:G_1|\leqq r!\leqq n!$). Furthermore, $A$ is central in $G_1$, i.e. $A\leqq \Z(G_1)$. 
To see this, notice that on an arbitrary fixed eigenspace $V_i$ $(1\leqq i\leqq r)$ $A$ acts by scalar matrices in such a way that all scalars are drawn from the set of the roots of unity. Since $G_1$ leaves $V_i$ invariant by definition and 
$\Imag(G_1\to\Aut K)$ fixes all roots of unity, our claim follows. After replacing $G$ with the bounded index subgroup $G_1$, we can assume that $A\leqq \Z(G)$.\\
As $A$ is a central subgroup of $G$, we can consider the quotient group $\overline{G}=G/A$. By Proposition \ref{CE}, we only need to prove that $\overline{G}$ has a bounded index nilpotent subgroup of class at most $c$.
Our strategy will be that, first we prove that $\overline{G}$ has a bounded index nilpotent subgroup of class at most $(c+1)$, then we will apply Lemma \ref{DN}.\\ 
Let $\overline{N}=N/A$, and consider the short exact sequence of groups
$$1\to \overline{N}\to\overline{G}\to\Gamma\to 1.$$ 
The number of elements of $\overline{N}$ is bounded by $J(n)$, by the definition of $A$, and $\Gamma$ is nilpotent of class at most $c$, by the definition of $G$.\\
$\overline{G}$ acts on $\overline{N}$ by conjugation, and the kernel of this action is the centralizer group $\Cent_{\overline{G}}(\overline{N})=\{g\in\overline{G}|\; ng=gn\;\forall n\in\overline{N}\}$. 
Therefore $\overline{G}/\Cent_{\overline{G}}(\overline{N})$ embeds into the automorphism group of $\overline{N}$ which has cardinality at most $J!$. Hence $\Cent_{\overline{G}}(\overline{N})$ has bounded index in $\overline{G}$. 
Hence, after replacing $\overline{G}$ with $\Cent_{\overline{G}}(\overline{N})$,  $\overline{N}$ with $\overline{N}\cap \Cent_{\overline{G}}(\overline{N})$ and $\Gamma$  with the image group $\Imag(\Cent_{\overline{G}}(\overline{N})\to \Gamma)$,
we can assume that  $\overline{G}$ is the central extension of the Abelian group $\overline{N}$ and nilpotent group $\Gamma$ whose nilpotency class is at most $c$. 
Therefore  we can assume that $\overline{G}$ is nilpotent of class at most $(c+1)$ (Proposition \ref{CE}).\\
Notice that $\gamma_c(\overline{G})$ maps to $\gamma_c(\Gamma)=1$, which implies that the former group is contained in $\overline{N}$. So $|\gamma_c(\overline{G})|\leqq|\overline{N}|\leqq J$. 
Hence we are in the position to apply Lemma \ref{DN}, which finishes the proof.
\end{proof}

\begin{rem}
\label{NoB}
In the above proof we only used the assumption that $G$ can be generated by $m$ elements via Lemma \ref{DN}. So if we omit this condition from Theorem \ref{groupmain}, we can still prove that there exists a constant 
$D=D(n)\in\mathbb{Z}^+$, only depending on $n$ (not even on $c$), such that if $G$ belongs to the corresponding family of groups, then $G$ contains a nilpotent subgroup $H\leqq G$ with nilpotency class at most $(c+2)$ and with index at most $D$. 
\end{rem}

\section{Proof of the Main Theorem}
\label{PMT}
Using the techniques developed in the previous sections, we will prove our main theorem.

\begin{thm}
\label{AlmostMain}
Fix a non-uniruled complex variety $Z_0$. Let $F_{Z_0}$ be the collection of 5-tuples $(X, Z,\phi, G, e)$, where
\begin{itemize}
\item $X$ is a complex variety,
\item $Z$ is a complex variety, which is birational to $Z_0$,
\item $\phi: X\dashrightarrow Z$ is a dominant rational map such that there exist open subvarieties $X_1$ of $X$ and $Z_1$ of $Z$ such that 
$\phi$ descends to a morphism between them $\phi_1:X_1\to Z_1$ with rationally connected fibres,
\item $G\leqq \Bir(X)$ is a finite group of the birational automorphism group of $X$, which also acts by birational automorphisms on $Z$ in such a way that $\phi$ is $G$-equivariant,
\item $e\in\mathbb{N}$ is the relative dimension $e=\dim X-\dim Z_0$.
\end{itemize}
Then the following claims hold.
\begin{itemize}
\item
There exist constants $\{m_{Z_0}(e)\in\mathbb{Z}^+|\,e\in\mathbb{N}\}$, only depending on the birational class of $Z_0$, such that if the 5-tuple $(X,Z,\phi, G, e)$ belongs to $F_{Z_0}$, then
$G$ can be generated by $m_{Z_0}(e)$ elements.
\item
There exist constants $\{J_{Z_0}(e)\in\mathbb{Z}^+|\,e\in\mathbb{N}\}$, only depending on the birational class of $Z_0$, such that if the 5-tuple $(X,Z,\phi, G, e)$ belongs to $F_{Z_0}$, then
$G$ has a nilpotent subgroup $H\leqq G$ of nilpotency class at most $(e+1)$ and index at most $J_{Z_0}(e)$.
\end{itemize}
\end{thm}

\begin{proof}
(Proof of the First Claim) Let  $(X, Z,\phi, G, e)$ be an arbitrary 5-tuple belonging to $F_{Z_0}$. By Lemma \ref{reg} we can assume that both $X$ and $Z$ are smooth projective varieties, and $G$ acts on them by regular automorphisms. 
Let $\rho$ be the generic point of $Z$, and let $X_\rho$ be the generic fibre of $\phi$. $X_\rho$ is a rationally connected variety  of dimension $e$ over the function field $K(Z)$.\\
Let $G_\rho\leqq G$ be the maximal subgroup of $G$ acting fibrewise. $G_\rho$ has a natural faithful action on $X_\rho$, while $G/G_\rho=G_Z$ has a natural faithful action on $Z$. 
This gives a short exact sequence of groups
\[1\to G_\rho\to G\to G_Z\to 1.\]
By Theorem \ref{bgrc} there exists a constant $m_1(e)$, only depending on $e$, such that $G_\rho$ can be generated by $m_1(e)$ elements. 
By Theorem \ref{bfsg} there exists a constant $m_2(Z)$, only depending on the birational class of $Z$, such that $G_Z$ can be generated by $m_2(Z)$ elements.
So $G$ can be generated by $m_{Z_0}(e)=m_1(e)+m_2(Z)$ elements. Since $m_{Z_0}(e)$ only depends on $e$ and the birational class of $Z_0$, this finishes the proof of the first claim.\\
(Proof the Second Claim) We will apply induction on $e$. If $e=0$, then $X$ and $Z_0$ are birational, hence $G\leqq \Bir(Z_0)$ and the claim of the theorem follows from Theorem \ref{nu}. 
So we can assume that $e>0$ and the claim of the theorem holds if the relative dimension is strictly smaller than $e$.\\
Let $(X, Z,\phi, G, e)$ be a 5-tuple belonging to $F_{Z_0}$.
After regularizing $\phi$ in the sense of Lemma \ref{reg}, we may assume that $X$ and $Z$ are smooth projective varieties, $G$ acts on them by regular automorphisms 
and $\phi$ is a $G$-equivariant (projective) morphism.\\
Hence by Theorem \ref{MMP}, we can run a relative $G$-equivariant MMP on $\phi: X\to Z$. It results a $G$-equivariant commutative diagram
\[
\xymatrix{
X \ar@{-->}[r]^\cong \ar[rd]_{\phi} & W \ar[r]^{\varrho}\ar[d] & Y \ar[ld]^{\psi}\\
& Z
}
\]
where $\varrho:W\to Y$ is a Mori fibre space and $\psi: Y\to Z$ is a dominant morphism with rationally connected general fibres (as so does $\phi$). 
Let $H$ be the image of $G\to\Aut_{\mathbb{C}}(Y)$, and let $f$ be the relative dimension $f=\dim Y-\dim Z$. The 5-tuple $(Y, Z,\psi, H, f)$ clearly belongs to $F_{Z_0}$.
Moreover, since $f<e$, we can use the inductive hypothesis. Let $H_1\leqq H$ be the nilpotent subgroup of nilpotency class at most $(f+1)$ and index at most $J_{Z_0}(f)$. After replacing $H$ with its bounded index subgroup $H_1$ 
(and $G$ with the preimage of $H_1$), we can assume that $H$ is nilpotent of class at most $e$.\\
Let $\eta\cong\Spec K(Y)$ be the generic point of $Y$, and let $W_\eta$ be the generic fibre of $\varrho$. Since $\varrho:W\to Y$ is a Mori fibre space, $W_\eta$ is a Fano variety over $K(Y)$ with (at worst) terminal singularities. 
Furthermore, $G$ acts on the structure morphism $W_\eta\to\Spec K(Y)$ equivariantly by scheme automorphisms. Hence we can apply Proposition \ref{Fano}, and we can embed $G$ to $\KL(n,K(Y))\cong\GL(n,K(Y))\rtimes \Aut K(Y)$ where
$n=n(e)$ only depends on $e$ (since $\dim W_\eta\leqq e$). 
Moreover, the image group $\Gamma=\Imag(G\hookrightarrow\KL(n,K(Y))\twoheadrightarrow \Aut K(Y))$ corresponds to the $G$-action on $\Spec K(Y)$, therefore it corresponds to the $H$-action on $Y$.
Hence $\Gamma$ fixes all roots of unity, as $Y$ is a complex variety, and $\Gamma$ is nilpotent of class at most $e$, as so does $H$. 
Furthermore, by the first claim of the theorem, every subgroup of $G$ can be generated by $m=m_{Z_0}(e)$ elements (where $m$ only depends on $e$ and the birational class of $Z_0$). 
So we are in the position to apply Theorem \ref{groupmain} to the group $G$, which finishes the proof. 
\end{proof}

\begin{rem}
\label{NoB2}
In accordance with Remark \ref{NoB}, we need to consider bounds on the number of generators of finite subgroups of the birational automorphism group to give a more accurate bound on the nilpotency class.
\end{rem}

To close our article, we prove our main theorem.

\begin{proof}[Proof of Theorem \ref{main}]
Let $X$ be a $d$ dimensional complex variety. We can assume that $X$ is smooth and projective. We can also assume that $X$ is non-uniruled by Theorem \ref{nu}.
Let $G\leqq\Bir(X)$ be an arbitrary finite subgroup of the birational automorphism group of $X$. Let $\phi: X\dashrightarrow Z$ be the MRC fibration, and let $e=\dim X-\dim Z$ be the relative dimension.
By the functoriality of the MRC fibration (Theorem $5.5$ of Chapter $4$ in \cite{Ko96}), $G$ acts on the base $Z$ by birational automorphisms making the rational map $\phi$ $G$-equivariant.  
Hence the 5-tuple $(X,Z,\phi,G, e)$ belongs to the collection $F_Z$ defined in the previous theorem. Therefore $G$ has a nilpotent subgroup of class at most $(e+1)$ and index at most $J_Z(e)$. 
Since $e<d$ (as $X$ is non-uniruled), moreover the relative dimension $e$ and the birational class of the base $Z$ only depends on the birational class of $X$, the theorem follows.
\end{proof}

\end{document}